\documentclass{amsart}
\usepackage[latin1]{inputenc}
\usepackage{amssymb,amsmath,amsthm}
\usepackage{amsfonts}
\usepackage{ifthen}
\usepackage{paralist}

\newcommand{\vect}[1]{\ensuremath{\mathbf{#1}}}
\newcommand{\card}[1]{\ensuremath{\lvert{#1}\rvert}}
\newcommand{\subf}[1][]{\ifthenelse{\equal{#1}{}}{\ensuremath{\leq}}{\ensuremath{\leq_{#1}}}}
\newcommand{\propsubf}[1][]{\ifthenelse{\equal{#1}{}}{\ensuremath{<}}{\ensuremath{{<}_{#1}}}}
\newcommand{\subc}[1][]{\ifthenelse{\equal{#1}{}}{\ensuremath{\preccurlyeq}}{\ensuremath{\preccurlyeq_{#1}}}}
\newcommand{\fequiv}[1][]{\ifthenelse{\equal{#1}{}}{\ensuremath{\equiv}}{\ensuremath{\equiv_{#1}}}}
\newcommand{\cl}[1]{\ensuremath{\mathcal{#1}}}

\theoremstyle{plain}
\newtheorem{theorem}{Theorem}[section]
\newtheorem{lemma}[theorem]{Lemma}
\newtheorem{corollary}[theorem]{Corollary}

\theoremstyle{remark}
\newtheorem{remark}[theorem]{Remark}

\begin{document}
\title{A note on minors determined by clones of semilattices}
\author{Erkko Lehtonen}
\address{Computer Science and Communications Research Unit \\
University of Luxembourg \\
6, rue Richard Coudenhove-Kalergi \\
L--1359 Luxembourg \\
Luxembourg}
\email{erkko.lehtonen@uni.lu}
\date{\today}
\begin{abstract}
The $\mathcal{C}$-minor partial orders determined by the clones generated by a semilattice operation (and possibly the constant operations corresponding to its identity or zero elements) are shown to satisfy the descending chain condition.
\end{abstract}

\maketitle

\section{Introduction}

This paper is a study of substitution instances of functions of several arguments when the inner functions are taken from a prescribed set of functions. Such an idea has been studied by several authors. Henno~\cite{Henno} generalized Green's relations to Menger algebras (essentially, abstract clones) and described Green's relations on the set of all operations on $A$ for each set $A$. Harrison~\cite{Harrison} considered two Boolean functions to be equivalent if they are substitution instances of each other with respect to the general linear group $\mathrm{GL}(n, \mathbb{F}_2)$ or the affine linear group $\mathrm{AGL}(n, \mathbb{F}_2)$, where $\mathbb{F}_2$ denotes the two-element field. In~\cite{Wang,WW}, a Boolean function $f$ is defined to be a minor of another Boolean function $g$, if ad only if $f$ can be obtained from $g$ by substituting for each variable of $g$ a variable, a negated variable, or one of the constants $0$ or $1$. Further variants of the notion of minor can be found in~\cite{CP,EFHH,FH,Pippenger,Zverovich}.

These ideas are unified and generalized by the notions of $\cl{C}$-minor and $\cl{C}$-equiv\-a\-lence, which first appeared in print in~\cite{LehtonenULM}. More precisely, let $A$ be a nonempty set, and let $f \colon A^n \to A$ and $g \colon A^m \to A$ be operations on $A$. Let $\cl{C}$ be a set of operations on $A$. We say that $f$ is a $\cl{C}$-minor of $g$, if $f = g(h_1, \dots, h_m)$ for some $h_1, \dots, h_m \in \cl{C}$, and we say that $f$ and $g$ are $\cl{C}$-equivalent if $f$ and $g$ are $\cl{C}$-minors of each other. If $\cl{C}$ is a clone, then the $\cl{C}$-minor relation is a preorder and it induces a partial order on the $\cl{C}$-equivalence classes. For background and basic results on $\cl{C}$-minors and $\cl{C}$-equivalences, see~\cite{LehtonenULM,LSdisc,LSff,LSff3}.

In this paper, we study the $\cl{C}$-minors and $\cl{C}$-equivalences induced by the clones generated by semilattice opeations (and possibly some constants). Our main result (Theorem~\ref{thm:main}) asserts that if $(A; \wedge)$ is a semilattice, then for the clone $\cl{C} = \langle \wedge \rangle$ generated by $\wedge$, the induced $\cl{C}$-minor partial order satisfies the descending chain condition. Furthermore, if $(A; \wedge)$ has an identity element $1$ and a zero element $0$, then this property is enjoyed by all clones $\cl{C}$ such that $\langle \wedge \rangle \subseteq \cl{C} \subseteq \langle \wedge, 0, 1 \rangle$. These results find an application in~\cite{LN}, in which the clones of Boolean functions are classified according to certain order-theoretical properties of their induced $\cl{C}$-minor partial orders.

\section{Clones, $\cl{C}$-minors and $\cl{C}$-decompositions}

\subsection{Operations and clones}

Throughout this paper, for an integer $n \geq 1$, we denote $[n] := \{1, \dots, n\}$.
Let $A$ be a fixed nonempty base set. An \emph{operation} on $A$ is a map $f \colon A^n \to A$ for some integer $n \geq 1$, called the \emph{arity} of $f$. We denote the set of all $n$-ary operations on $A$ by $\cl{O}_A^{(n)}$, and we denote by $\cl{O}_A := \bigcup_{n \geq 1} \cl{O}_A^{(n)}$ the set of all operations on $A$. The $i$-th $n$-ary \emph{projection} ($1 \leq i \leq n$) is the operation $(a_1, \dotsc, a_n) \mapsto a_i$, and it is denoted by $x_i^{(n)}$, or simply by $x_i$ when the arity is clear from the context.

If $f \in \cl{O}_A^{(n)}$ and $g_1, \dotsc, g_n \in \cl{O}_A^{(m)}$, then the \emph{composition} of $f$ with $g_1, \dotsc, g_n$, denoted $f(g_1, \dotsc, g_n)$ is the $m$-ary operation defined by
\[
f(g_1, \dotsc, g_n)(\vect{a}) = f \bigl( g_1(\vect{a}), \dotsc, g_n(\vect{a}) \bigr)
\]
for all $\vect{a} \in A^m$.

Let $\cl{C} \subseteq \cl{O}_A$. The \emph{$n$-ary part of $\cl{C}$} is the set $\cl{C}^{(n)} := \cl{C} \cap \cl{O}_A^{(n)}$ of $n$-ary members of $\cl{C}$.
A \emph{clone} on $A$ is a subset $\cl{C} \subseteq \cl{O}_A$ that contains all projections and is closed under composition, i.e., $f(g_1, \dots, g_n) \in \cl{C}$ whenever $f, g_1, \dots, g_n \in \cl{C}$ and the composition is defined.

The clones on $A$ constitute a complete lattice under inclusion. Therefore, for each set $F \subseteq \cl{O}_A$ of operations there exists a smallest clone that contains $F$, which will be denoted by $\langle F \rangle$ and called the \emph{clone generated by $F$.} See~\cite{DW,Lau,Szendrei} for general background on clones.

\subsection{$\cl{C}$-minors}

Let $\cl{C} \subseteq \cl{O}_A$, and let $f, g \in \cl{O}_A$. We say that $f$ is a \emph{$\cl{C}$-minor} of $g$, if $f = g(h_1, \dotsc, h_m)$ for some $h_1, \dotsc, h_m \in \cl{C}$, and we denote this fact by $f \subf[\cl{C}] g$. We say that $f$ and $g$ are \emph{$\cl{C}$-equivalent,} denoted $f \fequiv[\cl{C}] g$, if $f$ and $g$ are $\cl{C}$-minors of each other.

The $\cl{C}$-minor relation $\subf[\cl{C}]$ is a preorder (i.e., a reflexive and transitive relation) on $\cl{O}_A$ if and only if $\cl{C}$ is a clone. If $\cl{C}$ is a clone, then the $\cl{C}$-equivalence relation $\fequiv[\cl{C}]$ is an equivalence relation on $\cl{O}_A$, and, as for preorders, $\subf[\cl{C}]$ induces a partial order $\subc[\cl{C}]$ on the quotient $\cl{O}_A / {\fequiv[\cl{C}]}$. It follows from the definition of $\cl{C}$-minor, if $\cl{C}$ and $\cl{K}$ are clones such that $\cl{C} \subseteq \cl{K}$, then ${\subf[\cl{C}]} \subseteq {\subf[\cl{K}]}$ and ${\fequiv[\cl{C}]} \subseteq {\fequiv[\cl{K}]}$. For further background and properties of $\cl{C}$-minor relations, see \cite{LehtonenULM,LSdisc,LSff,LSff3}.

\subsection{$C$-decompositions}

Let $\cl{C}$ be a clone on $A$, and let $f \in \cl{O}_A^{(n)}$. If $f = g(\phi_1, \dotsc, \phi_m)$ for some $g \in \cl{O}_A^{(m)}$ and $\phi_1, \dotsc, \phi_m \in \cl{C}$, then we say that the $(m+1)$-tuple $(g, \phi_1, \dotsc, \phi_m)$ is a \emph{$\cl{C}$-decomposition} of $f$. We often avoid referring explicitly to the tuple and we simply say that $f = g(\phi_1, \dotsc, \phi_m)$ is a $\cl{C}$-decomposition. Clearly, there always exists a $\cl{C}$-decomposition of every $f$ for every clone $\cl{C}$, because $f = f(x_1^{(n)}, \dotsc, x_n^{(n)})$ and projections are members of every clone. A $\cl{C}$-decomposition of a nonconstant function $f$ is \emph{minimal} if the arity $m$ of $g$ is the smallest possible among all $\cl{C}$-decompositions of $f$. This smallest possible $m$ is called the \emph{$\cl{C}$-degree} of $f$, denoted $\deg_{\cl{C}} f$. We agree that the $\cl{C}$-degree of any constant function is $0$.

\begin{lemma}
\label{lemma:Cdeg}
If $f \subf[\cl{C}] g$, then $\deg_{\cl{C}} f \leq \deg_{\cl{C}} g$.
\end{lemma}
\begin{proof}
Let $\deg_{\cl{C}} g = m$, and let $g = h(\gamma_1, \dotsc, \gamma_m)$ be a minimal $\mathcal{C}$-decomposition of $g$. Since $f \subf[\cl{C}] g$, there exist $\phi_1, \dotsc, \phi_n \in \cl{C}$ such that $f = g(\phi_1, \dotsc, \phi_n)$. Then
\[
f = h(\gamma_1, \dotsc, \gamma_m)(\phi_1, \dotsc, \phi_n)
= h(\gamma_1(\phi_1, \dotsc, \phi_n), \dotsc, \gamma_m(\phi_1, \dotsc, \phi_n)),
\]
and since $\gamma_i(\phi_1, \dotsc, \phi_n) \in \cl{C}$ for $1 \leq i \leq m$, we have that $(h, \gamma_1(\phi_1, \dotsc, \phi_n), \dotsc,\linebreak[1] \gamma_m(\phi_1, \dotsc, \phi_n))$ is a $\cl{C}$-decomposition of $f$, not necessarily minimal, so $\deg_{\cl{C}} f \leq m$.
\end{proof}

An immediate consequence of Lemma \ref{lemma:Cdeg} is that $\cl{C}$-equivalent functions have the same $\cl{C}$-degree.

Let $(\phi_1, \dotsc, \phi_m)$ be an $m$-tuple ($m \geq 2$) of $n$-ary operations on $A$. If there is an $i \in \{1, 2, \dotsc, m\}$ and $g \colon A^{m-1} \to A$ such that
\[
\phi_i = g(\phi_1, \dotsc, \phi_{i-1}, \phi_{i+1}, \dotsc, \phi_m),
\]
we say that the $m$-tuple $(\phi_1, \dotsc, \phi_m)$ is \emph{functionally dependent.} Otherwise we say that $(\phi_1, \dotsc, \phi_m)$ is \emph{functionally independent.} We often omit the $m$-tuple notation and simply say that $\phi_1, \dotsc, \phi_m$ are functionally dependent or independent.

\begin{remark}
\label{rem:fi}
Every $m$-tuple containing a constant function is functionally dependent. Also if $f_i = f_j$ for some $i \neq j$, then $f_1, \dotsc, f_n$ are functionally dependent.
\end{remark}

\begin{lemma}
\label{lemma:functindep}
If $(g, \phi_1, \dotsc, \phi_m)$ is a minimal $\cl{C}$-decomposition of $f$, then $\phi_1, \dotsc, \phi_m$ are functionally independent.
\end{lemma}
\begin{proof}
Suppose, on the contrary, that $\phi_1, \dotsc, \phi_m$ are functionally dependent. Then there is an $i$ and an $h \colon A^{m-1} \to A$ such that $\phi_i = h(\phi_1, \dotsc, \phi_{i-1}, \phi_{i+1}, \dotsc, \phi_m)$. Then
\begin{multline*}
f = g(\phi_1, \dotsc, \phi_{i-1}, h(\phi_1, \dotsc, \phi_{i-1}, \phi_{i+1}, \dotsc, \phi_m), \phi_{i+1}, \dotsc, \phi_m) \\
= g(x^{(m-1)}_1, \dotsc, x^{(m-1)}_{i-1}, h, x^{(m-1)}_i, \dotsc, x^{(m-1)}_{m-1})(\phi_1, \dotsc, \phi_{i-1}, \phi_{i+1}, \dotsc, \phi_m),
\end{multline*}
which shows that $(g(x_1, \dotsc, x_{i-1}, h, x_i, \dotsc, x_{m-1}), \phi_1, \dotsc, \phi_{i-1}, \phi_{i+1}, \dotsc, \phi_m)$ is a $\cl{C}$-decomposition of $f$, contradicting the minimality of $(g, \phi_1, \dotsc, \phi_m)$.
\end{proof}

\section{$\cl{C}$-minors determined by clones of semilattices}

In this section we will prove our main result, namely Theorem~\ref{thm:main}. It will find an application in~\cite{LN} where the clones of Boolean functions are classified according to certain order-theoretical properties that their induced $\cl{C}$-minor partial orders enjoy.

An operation $\wedge$ on $A$ is called a \emph{semilattice operation,} if for all $x, y, z \in A$, the following identities hold:
\[
x \wedge (y \wedge z) = (x \wedge y) \wedge z, \qquad
x \wedge y = y \wedge x ,\qquad
x \wedge x = x,
\]
i.e., $\wedge$ is associative, commutative and idempotent.

A partial order $(P; \leq)$ is said to satisfy the \emph{descending chain condition,} or it is called \emph{well-founded,} if it contains no infinite descending chains, i.e., given any sequence of elements of $P$
\[
\cdots \leq a_3 \leq a_2 \leq a_1,
\]
there exists a positive integer $n$ such that
\[
a_n = a_{n+1} = a_{n+2} = \cdots.
\]

\begin{theorem}
\label{thm:main}
Let $\cl{S}$ be the clone generated by a semilattice operation $\wedge$ on $A$. Then the $\cl{S}$-minor partial order $\subc[\cl{S}]$ satisfies the descending chain condition.
\end{theorem}
\begin{proof}
Let $(\phi_1, \dots, \phi_m) \in (\cl{S}^{(n)})^m$. Then, for $1 \leq j \leq m$, $\phi_j$ is of the form
\begin{equation}
\phi_j = \bigwedge_{i \in \Phi_j} x_i^{(n)}
\label{eq:conj}
\end{equation}
for some $\emptyset \neq \Phi_j \subseteq [n]$. For $1 \leq i \leq n$, denote
\begin{equation}
X_i := \{j \in [m] : i \in \Phi_j\},
\label{eq:Xi}
\end{equation}
and let $X(\phi_1, \dots, \phi_m) := \{X_1, \dots, X_n\} \subseteq \mathcal{P}([m])$. It follows from the definitions of $\Phi_j$ and $X_i$ that
\begin{equation}
j \in X_i
\quad\iff\quad
i \in \Phi_j.
\label{eq:iff}
\end{equation}
Correspondingly, for any $\emptyset \neq E \subseteq \mathcal{P}([m])$, denote $\Psi_E := (\psi_1, \dots, \psi_m)$, where $\psi_j \in \cl{S}^{(\card{E})}$ is given by
\[
\psi_j = \bigwedge_{j \in S \in E} x_{\sigma_E(S)},
\]
where $\sigma_E \colon E \to [\card{E}]$ is any fixed bijection.

Let $(g, \phi_1, \dotsc, \phi_m)$ be a $\cl{S}$-decomposition of $f \colon A^n \to A$. Then each $\phi_j$ is of the form \eqref{eq:conj} for some $\emptyset \neq \Phi_j \subseteq [n]$. Let $E := X(\phi_1, \dots, \phi_m)$, $(\psi_1, \dots, \psi_m) := \Psi_E$, and let $f' = g(\psi_1, \dotsc, \psi_m)$. We will show that $f \fequiv[\cl{S}] f'$.

As in~\eqref{eq:Xi}, for $1 \leq i \leq n$, let $X_i = \{j \in [m] : i \in \Phi_j\}$. Let $\pi \colon [n] \to [\card{E}]$ be defined as $\pi(i) := \sigma_E(X_i)$. Then
\begin{multline*}
f(x_{\pi(1)}, \dotsc, x_{\pi(n)})
= g(\phi_1, \dotsc, \phi_m)(x_{\pi(1)}, \dotsc, x_{\pi(n)}) \\
= g(\phi_1(x_{\pi(1)}, \dotsc, x_{\pi(n)}), \dotsc, \phi_m(x_{\pi(1)}, \dotsc, x_{\pi(n)}))
= g(\psi_1, \dotsc, \psi_m) = f',
\end{multline*}
where the second to last equality holds because for $1 \leq j \leq m$,
\[
\phi_j(x_{\pi(1)}, \dotsc, x_{\pi(n)})
= \bigwedge_{i \in \Phi_j} x_{\pi(i)}
= \bigwedge_{i \in \Phi_j} x_{\sigma(X_i)}
= \bigwedge_{j \in S \in E} x_{\sigma(S)}
= \psi_j.
\]
Since all projections are members of $\cl{S}$, we have that $f' \subf[\cl{C}] f$.
On the other hand, for $1 \leq j \leq \card{E}$, let $\Xi_j := \{i \in [n] : X_i = \sigma_E^{-1}(j)\}$, and let
\[
\xi_j := \bigwedge_{i \in \Xi_j} x_i
\]
It is easy to see that $\Xi_j \neq \emptyset$; hence $\xi_j \in \cl{S}$. Then
\begin{multline*}
f'(\xi_1, \dotsc, \xi_{\card{E}})
= g(\psi_1, \dotsc, \psi_m)(\xi_1, \dotsc, \xi_{\card{E}}) \\
= g(\psi_1(\xi_1, \dotsc, \xi_{\card{E}}), \dotsc, \psi_m(\xi_1, \dotsc, \xi_{\card{E}}))
= g(\phi_1, \dotsc, \phi_m)
= f,
\end{multline*}
where the second to last equality holds because for $j = 1, \dotsc, m$,
\begin{multline*}
\psi_j(\xi_1, \dotsc, \xi_{\card{E}})
= \Bigl( \bigwedge_{j \in S \in E} x_{\sigma_E(S)} \Bigr) (\xi_1, \dotsc, \xi_{\card{E}})
= \bigwedge_{j \in S \in E} \xi_{\sigma_E(S)} \\
= \bigwedge_{j \in S \in E} \Bigl( \bigwedge_{i \in \Xi_{\sigma_E(S)}} x_i \Bigr)
= \bigwedge_{j \in S \in E} \Bigl( \mathop{\bigwedge_{i \in [n]}}_{X_i = S} x_i \Bigr)
= \bigwedge_{i \in \Phi_j} x_i
= \phi_j.
\end{multline*}
Here, the third last equality holds, because $\Xi_{\sigma_E(S)} = \{i \in [n] : X_i = S\}$, and the second last equality holds by~\eqref{eq:iff} and the associativity, commutativity and idempotency of $\wedge$. Since $\xi_j \in \cl{S}$, we have that $f \subf[\cl{C}] f'$. We conclude that $f \fequiv[\cl{C}] f'$, as desired.

\textit{Claim.} If $f_1 = g(\phi_1, \dotsc, \phi_m)$ and $f_2 = g(\varphi_1, \dotsc, \varphi_m)$ are $\cl{S}$-decompositions and $X(\phi_1, \dotsc, \phi_m) = X(\varphi_1, \dotsc, \varphi_m)$, then $f_1 \fequiv[\cl{S}] f_2$.

\textit{Proof of the claim.} Let $(\psi_1, \dotsc, \psi_m) := \Psi_{X(\phi_1, \dotsc, \phi_m)}$ ($= \Psi_{X(\varphi_1, \dotsc, \varphi_m)}$), and let $f' = g(\psi_1, \dotsc, \psi_m)$. It follows from what was shown above that $f_1 \fequiv[\cl{S}] f' \fequiv[\cl{S}] f_2$. The claim follows by the transitivity of $\fequiv[\cl{S}]$.
\qquad$\diamond$

To finish the proof that $\subc[\cl{S}]$ satisfies the descending chain condition, assume that $f_1 \propsubf[\cl{S}] f_2$, $f_2 = g(\phi_1, \dotsc, \phi_m)$ is a minimal $\cl{S}$-decomposition, and $f_1 = f_2(h_1, \dotsc, h_n)$ for some $h_1, \dotsc, h_n \in \cl{S}$. For $i = 1, \dotsc, m$, denote $\phi'_i = \phi_i(h_1, \dotsc, h_n)$, so that $f_1 = g(\phi'_1, \dotsc, \phi'_m)$. By Lemma \ref{lemma:Cdeg}, either $\deg_{\cl{S}} f_1 < \deg_{\cl{S}} f_2$, or $\deg_{\cl{S}} f_1 = \deg_{\cl{S}} f_2$ and $X(\phi_1, \dotsc, \phi_m) \neq X(\phi'_1, \dotsc, \phi'_m)$. Since $\cl{S}$-degrees are nonnegative integers and $\mathcal{P}([m])$ is a finite set, there are only a finite number of $\fequiv[\cl{S}]$-classes preceding the $\fequiv[\cl{S}]$-class of $f_2$ in the $\cl{S}$-minor partial order $\subc[\cl{S}]$. This completes the proof of the theorem.
\end{proof}

\begin{corollary}
Assume that a semilattice $(A; \wedge)$ has identity and zero elements $1$ and $0$, respectively. Let $\cl{C}$ be a clone on $A$ such that $\langle \wedge \rangle \subseteq \cl{C} \subseteq \langle \wedge, 0, 1 \rangle$. Then the $\cl{C}$-minor partial order $\subc[\cl{C}]$ satisfies the descending chain condition.
\end{corollary}

\begin{proof}
The proof of Theorem~\ref{thm:main} in fact shows that $\subc[\cl{C}]$ satisfies the descending chain condition. For, in this case $\cl{C} \setminus \cl{S}$ contains only constant operations. Remark~\ref{rem:fi} and Lemma \ref{lemma:functindep} guarantee that $f = g(h_1, \ldots, h_m)$ is a minimal $\cl{S}$-decomposition if and only if it is a minimal $\cl{C}$-decomposition, and since $\cl{S} \subseteq \cl{C}$, $\cl{S}$-equivalence implies $\cl{C}$-equivalence.
\end{proof}

\end{document}